\documentclass[12pt,reqno,oneside]{amsart}
\usepackage{amsmath, amssymb}
\usepackage[mathcal]{eucal}
\usepackage{hyperref}

\numberwithin{equation}{section} 
\newtheorem{theorem}{Theorem}[section]
\newtheorem{corollary}[theorem]{Corollary}
\newtheorem{lemma}[theorem]{Lemma}

\theoremstyle{definition}
\newtheorem{definition}[theorem]{Definition}

\theoremstyle{remark}
\newtheorem{remark}[theorem]{Remark}
\numberwithin{equation}{section}

\thanks{* Corresponding author, Email: gmsmoorthy@yahoo.com}
\begin{document}
\title [Subclasses of Bi-Univalent Functions... ]{Subclasses
of Bi-Univalent Functions Defined by \\ S\u al\u agean type $q-$
Difference Operator  }
\author{G. Murugusundaramoorthy$^{1,*}$ and  K. VIJAYA$^{2}$, }
\maketitle
\begin{center}
$^{*}$ Corresponding Author \\$^{1,2}$School of Advanced Sciences,\\ VIT University, \\ Vellore -
632014, India.\\ {\bf E-mail:~~{\it gmsmoorthy@yahoo.com,kvijaya@vit.ac.in,}}
\end{center}

\begin{abstract}
In this paper, we introduce and investigate a new subclass of the function
class $\Sigma$ of bi-univalent functions  defined in the open unit disk, which are associated with the  S\u al\u agean type $q-$ difference operator and satisfy some subordination conditions.
Furthermore, we find estimates on the Taylor-Maclaurin
coefficients $|a_2|$ and $|a_3|$ for functions in the new subclass introduced here. Several (known or new) consequences of the results are also pointed out. Further we obtain Fekete-Szeg$\ddot{o}$ inequality for the new function class.
\end{abstract}
\maketitle
%%%=====================
\section{Introduction} Let $\mathcal{A}$ denote the class of analytic functions of the form
\begin{equation}\label{c7e1}
f(z)=z+\sum\limits_{n=2}^{\infty}a_nz^n
\end{equation}
normalized by the conditions $f(0) = 0 = f'(0) - 1$ defined in the open unit disk $$\triangle = \{  z \in \mathbb{C} : |z| < 1  \}.$$ Let $\mathcal{S}$ be the subclass
of $\mathcal{A}$ consisting of functions of the form \eqref{c7e1} which are also
univalent in $\triangle.$ Let $\mathcal{S}^{*}(\alpha)$ and $\mathcal{K}(\alpha)$ denote the subclasses of  $\mathcal{S},$
consisting of starlike and convex functions of order $\alpha,\ 0 \leq \alpha < 1,$ respectively.
%----------------------------------------------------------------------------------------------------------
An analytic function $f$ is subordinate to an analytic function $g,$ written $f(z) \prec g(z),$
provided there is an analytic function $w$ defined on $\triangle$ with $w(0) = 0$ and
$|w(z)|< 1 $ satisfying $f(z) = g(w(z)).$ Ma and Minda \cite{ma} unified various subclasses
of starlike and convex functions for which either of the quantities $$\frac{z\ f'(z)}{f(z)}\quad  {\rm or}
\quad 1 +  \frac{z\ f''(z)}{f'(z)}$$ is subordinate to a more general superordinate function.
For this purpose, they considered an analytic function $\varphi$ with positive real part in
the unit disk $\triangle,~~ \varphi(0) = 1,~~ \varphi'(0) > 0 $ and $\varphi$ maps $\triangle$
onto a region starlike with respect to 1 and symmetric with respect to the real axis.
\iffalse
The class of Ma-Minda starlike functions consists of functions $f \in \mathcal{A}$ satisfying
the subordination $\frac{z\ f'(z)}{f(z)} \prec \varphi(z).$ Similarly, the class of Ma-Minda
convex functions consists of  functions $f \in \mathcal{A}$ satisfying the subordination
$1 +  \frac{z\ f''(z)}{f'(z)} \prec \varphi(z).$ A function $f$ is bi-starlike of Ma-Minda
type or bi-convex of Ma-Minda type if both $f$ and $f^{-1}$ are respectively Ma-Minda starlike
or convex. These classes are denoted respectively by $\mathcal{S}^{*}_{\Sigma}(\varphi)$ and
$\mathcal{K}_{\Sigma}(\varphi).$ In the sequel, it is assumed that $\varphi$ is an analytic function with positive real part in
the unit disk $\triangle,$ satisfying $\varphi(0) = 1, ~~\varphi'(0) > 0$ and $\varphi(\triangle)$
is symmetric with respect to the real axis.\fi In the sequel, it is assumed that such a function has a series expansion of the form
\begin{equation}\label{c7e3}
\phi(z) = 1 + B_{1} z + B_{2} z^2 + B_{3} z^3 + \cdots,\ \ \ (B_{1}> 0).
\end{equation}
In particular for the class of strongly starlike functions of order $\alpha(0< \alpha \leq 1),$ the function $\phi$ is given by
\begin{equation}\label{phi01}
\phi(z) = \left( \frac{1 + z}{1 - z} \right)^{\alpha} = 1 + 2 \alpha z + 2 \alpha^2 z^2 + \cdots\quad
(0 < \alpha \leq 1),
\end{equation}
which gives $B_{1} = 2 \alpha~~{\rm  and }~~B_{2} = 2 \alpha^2$ and
on the other hand, for the class of  starlike functions of order $\beta(0\leq \beta < 1),$
\begin{equation}\label{phi02}
\phi(z) =  \frac{1 + (1 - 2 \beta) z}{1 - z}  = 1 + 2 (1 - \beta) z + 2 (1 - \beta) z^2 + \cdots\quad
(0 \leq  \beta < 1),
\end{equation}
we have$B_{1} = B_{2} = 2(1 - \beta).$
\\
\par The Koebe one quarter theorem \cite{duren} ensures that the image of $\triangle$ under
every univalent function $f \in \mathcal{A}$ contains a disk of radius $\frac{1}{4}.$
Thus every univalent function $f$ has an inverse $f^{-1}$ satisfying $f^{-1}(f(z)) = z,
\ (z \in \triangle)$ and $f(f^{-1}(w)) = w\ (|w| < r_{0}(f), \ r_{0}(f)\geq \frac{1}{4}).$
A function $f \in \mathcal{A}$ is said to be bi-univalent in $\triangle$ if both
$f$ and $f^{-1}$ are univalent in $\triangle.$ Let $\Sigma$ denote the class of
bi-univalent functions defined in the unit disk $\triangle.$ Since $f \in \Sigma$
has the Maclaurian series given by \eqref{c7e1}, a computation shows that its inverse
$g = f^{-1}$ has the expansion
\begin{equation}\label{c7e2}
g(w) = f^{-1}(w) = w - a_{2} w^2 + (2a_{2}^2 - a_{3})w^3 + \cdots.
\end{equation}Several authors have  introduced and investigated subclasses of bi-univalent
functions  and obtained bounds for the initial coefficients (see \cite{bran-clu,bran-saha,DEN,lewin,hya,srivastava,xu1,xi}).
\par Quantum calculus is ordinary classical calculus without the notion of limits.It
defines \textit{q-calculus} and\textit{h-calculus}. Here $h$ ostensibly stands for Planckís constant,
while $q$ stands for quantum. Recently, the area of \textit{q-calculus} has attracted the
serious attention of researchers. This great interest is due to its application in
various branches of mathematics and physics.The application of \textit{q-calculus} was
initiated by Jackson\cite{fhjj}. He was the first to develop \textit{q-integral} and \textit{q-derivative} in a
systematic way. Later, geometrical interpretation of \textit{q-analysis} has been recognized
through studies on quantum groups. It also suggests a relation between integrable
systems and \textit{q-analysis}.A comprehensive study on applications of \textit{q-calculus} in
operator theory may be found in\cite{aga}. For the convenience, we provide
some basic definitions and concept details of \textit{q-calculus}
which are used in this paper. %We suppose throughout the paper that $0 < q < 1.$
\par For $0<q<1$  the Jackson's \textit{q-derivative}  of a function
$f\in \mathcal{A}$ is, by definition,  given as follows \cite{fhjj}
\begin{equation}\label{in5}
\mathcal{D}_{q}f(z) = \left\{\begin{array}{lcl}\dfrac{f(z) - f(qz)}{(1 -
q)z}&for &z \neq 0,\\ f'(0) &for& z=0,\end{array}\right.
\end{equation}
and$ \quad \mathcal{D}^{2}_{q}f(z) = \mathcal{D}_{q}(\mathcal{D}_{q}f(z)).$ From (\ref{in5}), we have
\begin{equation}\label{dqf}
\mathcal{D}_{q}f(z)=1+\sum\limits_{n=2}^{\infty}[n]_qa_{n}z^{n-1}
\end{equation}
where
\begin{equation}\label{in6}
[n]_q= \frac{1 - q^{n}}{1 - q},
\end{equation}
is sometimes called \textit{the basic number} $n$. If
$q\rightarrow1^{-}, [n]\rightarrow n$.For a function $h(z) =
z^{n},$ we obtain
$\mathcal{D}_{q}h(z)=\mathcal{D}_{q}z^{n}=\frac{1 - q^{n}}{1 - q}z^{n-1}=[n]z^{n-1},$ and
$\lim_{q\rightarrow 1^-}\mathcal{D}_{q}h(z)=\lim_{q\rightarrow 1^-}\left([n]z^{n-1}\right)= nz^{n-1} = h'(z), $
where $h'$ is the ordinary derivative.
Recently for $f\in \mathcal{A},$  Govindaraj and Sivasubramanian  \cite{siva} defined and discussed the S\u al\u agean \textit{q-differential operator} as given below:
 \begin{eqnarray}\label{dk}
 \mathcal{D}^0_q f(z)&=& f(z)\nonumber \\
  \mathcal{D}^1_q f(z)&=&z  \mathcal{D}_q f(z)\nonumber \\
  \mathcal{D}^k_q f(z)&=&z \mathcal{D}^k_q( \mathcal{D}_q^{k-1} f(z))\nonumber \\
    \mathcal{D}_q^k f(z)&=&z+ \sum\limits_{n=2}^{\infty}[n]^k_q a_nz^n\quad (k\in \mathbb{N}_0,z\in\Delta)
 \end{eqnarray}
We note that $\lim_q\rightarrow 1^-$
\begin{equation}\label{gsala}
   \mathcal{D}^k f(z)=z+ \sum\limits_{n=2}^{\infty}n^k a_nz^n\quad (k\in \mathbb{N}_0,z\in\Delta)
\end{equation} the familiar S\u al\u agean derivative\cite{salagean}.
\iffalse
\cite{}  differential operator   $$\mathcal{D}^{k} : \mathcal{A} \rightarrow \mathcal{A}$$
defined by
\begin{eqnarray}\label{}
\mathcal{D}^{0} f(z) &=& f(z),\nonumber\\ \mathcal{D}^{1} f(z) &=& \mathcal{D }f(z) = z f'(z), \nonumber \\
\mathcal{D}^{k} f(z) &=& \mathcal{D}(\mathcal{D}^{k-1} f(z)) = z (\mathcal{D}^{k-1} f(z))', \quad k \in \mathbb{N} = \{  1,2,3,\ldots  \}\nonumber\\
\mathcal{D}^{k} f(z) &=& z + \sum\limits_{n = 2}^{\infty} n^k a_{n} z^{n}, \quad k \in \mathbb{N}_{0} = \mathbb{N} \cup \{ 0 \},
\end{eqnarray}
\fi
\par In this paper, making use of
the S\u al\u agean \textit{q-differential operator}, for functions $g$ of the form (\ref{c7e2}) we  define
\begin{equation}\label{c7e2a}
\mathcal{D}^k g(w)  = w - a_{2}[2]_q^k w^2 + (2a_{2}^2 - a_{3})[3]_q^k w^3 + \cdots
\end{equation}
and introduce two new subclass of bi-univalent functions to obtain
the estimates on the coefficients $|a_{2}|$ and $|a_{3}|$ by Ma-Minda subordination.Further by using the  initial coefficient values  of $a_{2}$ and $a_{3}$ we also obtain Fekete-Szeg$\ddot{o}$ inequalities.
%===================================================================================================================

\section{Bi-Univalent function class $\mathcal{M}{\Sigma}_q^{k}(\lambda, \phi)$ }
\par In this section,  due to Vijaya et al \cite{gms} ,we introduce a subclass $\mathcal{M}{\Sigma}_q^{k}(\lambda, \phi)$ of
${\Sigma}$ and find estimate on the coefficients
$|a_{2}|$ and $|a_{3}|$ for the functions in this new subclass,  by subordination. Throughout our study, unless otherwise stated, we let
$$0\leq \lambda \leq 1;\quad 0 < q < 1; \quad  k\in \mathbb{N}_0 $$
\begin{definition}\label{c7def3.1}
For $0\leq \lambda \leq 1,$ a  function $f \in \Sigma$  of the form (\ref{c7e1}) is said to be in the class $\mathcal{M}{\Sigma}_q^{k}(\lambda, \phi)$
if the following subordination hold:
\begin{equation}\label{c7e3.1}
(1 - \lambda) \frac{\mathcal{D}_q^{k + 1} f(z)} { \mathcal{D}_q^{k} f(z)}+ \lambda \frac{\mathcal{D}_q^{k + 2} f(z)}{\mathcal{D}_q^{k + 1} f(z)} \prec \phi(z)
\end{equation}
and
\begin{equation}\label{c7e3.2}
(1 - \lambda) \frac{\mathcal{D}_q^{k + 1} g(w)} { \mathcal{D}_q^{k} g(w)}+ \lambda \frac{\mathcal{D}_q^{k + 2} g(w)}{\mathcal{D}_q^{k + 1} g(w)}  \prec \phi(w),
\quad
\end{equation}where $z,w \in \Delta$ and $g$ is given by (\ref {c7e2}).
\end{definition}

\begin{remark}\label{c7rem3.2}Suppose $f\in\Sigma.$ If $\lambda = 0,$ then
$\mathcal{M}{\Sigma}_q^{k}(\lambda, \phi)\equiv\mathcal{S}{\Sigma}_q^{k}(\phi):$ thus
 $f\in \mathcal{S}{\Sigma}_q^{k}(\phi)$ if the following subordination holds:
\begin{equation*}
\frac{\mathcal{D}_q^{k + 1} f(z)}{\mathcal{D}_q^{k} f(z)} \prec \phi(z)\qquad {\rm and} \qquad
\frac{\mathcal{D}_q^{k + 1} g(w)}{\mathcal{D}_q^{k} g(w)} \prec \phi(w),
\end{equation*}
where $z,w \in \Delta$ and $g$ is given by (\ref {c7e2}).
\end{remark}

\begin{remark}\label{c7rem3.3}
Suppose $f\in\Sigma.$  If $\lambda = 1,$ then
$\mathcal{M}{\Sigma}_q^{k}(\lambda, \phi)\equiv\mathcal{K}{\Sigma}_q^{k}(\phi):$ thus
 $f\in \mathcal{K}{\Sigma}_q^{k}(\phi)$ if the following subordination holds:
\begin{equation*}
\frac{\mathcal{D}_q^{k + 2} f(z)}{\mathcal{D}_q^{k + 1} f(z)} \prec \phi(z)\qquad {\rm and} \qquad
\frac{\mathcal{D}_q^{k + 2} g(w)}{\mathcal{D}_q^{k + 1} g(w)} \prec \phi(w),
\end{equation*}
where $z,w \in \Delta$ and $g$ is given by (\ref {c7e2}).
\end{remark}
\begin{remark}\label{gms3.3}
For $0\leq \lambda \leq 1~ {\rm and}~ k=0$ a  function $f \in \Sigma$  of the form (\ref{c7e1}) is said to be in the class $\mathcal{M}{\Sigma}_q^{k}(\lambda, \phi)$
if the following subordination hold:
\begin{equation}\label{gms3.1}
(1 - \lambda) \frac{z\mathcal{D}_q f(z)} {f(z)}+ \lambda \frac{\mathcal{D}_q(z\mathcal{D}_qf(z)}{\mathcal{D}_qf(z))} \prec \phi(z)
\end{equation}
and
\begin{equation}\label{gms3.2}
(1 - \lambda) \frac{z\mathcal{D}_q g(w)} {g(w)}+ \lambda \frac{\mathcal{D}_q(w\mathcal{D}_qg(w)}{\mathcal{D}_qg(w))}  \prec \phi(w),
\quad
\end{equation}where $z,w \in \Delta$ and $g$ is given by (\ref {c7e2}).

\end{remark}
\par It is of interest to  note that $\mathcal{M}{\Sigma}_q^{0}(0, \phi) = \mathcal{S}{\Sigma}_q^{*}(\phi),$~~$ \mathcal{M}{\Sigma}_q^{0}(1, \phi)
= \mathcal{K}{\Sigma}_q(\phi)$ new subclasses of $\Sigma$ associated with $q-$ difference operator not yet discussed sofar.
\par In order to prove our main results, we require the following Lemma:

\begin{lemma}\cite{pomm}\label{c7lem1.2}
{\it If a function $p \in \mathcal{P}$ is given by
$$p(z) = 1 + p_{1} z + p_{2} z^2 + \cdots\qquad (z \in \Delta),$$
then $$|p_{i}|\leqq 2\qquad (i\in \mathbb{N}),$$ where $\mathcal{P}$ is the family
of all functions $p,$ analytic in $\Delta,$ for which
$$p(0)=1 \qquad \textit{and} \qquad \Re\big( p(z) \big) > 0\qquad (z \in \Delta).$$}

\end{lemma}

\begin{theorem}\label{c7thm3.4}
Let $f$ given by \eqref{c7e1} be in the class $\mathcal{M}{\Sigma}_q^{k}(\lambda, \phi).$ Then
\begin{equation}\label{c7e3.3}
|a_{2}| \leq \frac{ B_{1} \sqrt{B_{1}} } { \sqrt{|(2(1+2\lambda) [3]_q^k - (1 + 3\lambda)[2]_q^{2 k}) B_{1}^2
+ (1 + \lambda)^2 (B_{1} - B_{2})[2]_q^{2 k} |} }
\end{equation}
and
\begin{equation}\label{c7e3.4}
|a_{3}| \leq   \frac{B_{1}}{2(1 + 2 \lambda) [3]_q^{k}}+ \left(\frac{B_{1}}{(1 + \lambda)[2]_q^{k}}\right)^2
\end{equation}where $0\leq \lambda \leq 1.$
\end{theorem}

\begin{proof}
Let $f \in \mathcal{M}{\Sigma}_q^{k}(\lambda, \phi)$ and $g = f^{-1}.$ Then there are analytic
functions $u, v : \triangle \longrightarrow \triangle,$ with $u(0) = 0 = v(0),$ satisfying
\begin{equation}\label{c7e2.3}
(1 - \lambda) \frac{\mathcal{D}_q^{k + 1} f(z)} { \mathcal{D}_q^{k} f(z)}+ \lambda \frac{\mathcal{D}_q^{k + 2} f(z)}{\mathcal{D}_q^{k + 1} f(z)} = \phi(u(z))
\end{equation}
and
\begin{equation}\label{c7e2.4}
(1 - \lambda) \frac{\mathcal{D}_q^{k + 1} g(w)} { \mathcal{D}_q^{k} g(w)}+ \lambda \frac{\mathcal{D}_q^{k + 2} g(w)}{\mathcal{D}_q^{k + 1} g(w)} = \phi(v(w)).
\end{equation}
Define the functions $p(z)$ and $q(z)$ by
\begin{equation*}
p(z) : = \frac{1 + u(z)}{1 - u(z)} = 1 + p_{1} z + p_{2} z^2 + \cdots
\end{equation*}
and
\begin{equation*}
q(z) : = \frac{1 + v(z)}{1 - v(z)} = 1 + q_{1} z + q_{2} z^2 + \cdots
\end{equation*}
or, equivalently,
\begin{equation}\label{c7e2.5}
u(z) : = \frac{p(z) - 1}{p(z) + 1} = \frac{1}{2}\left[ p_{1} z + \left(p_{2} - \frac{p_{1}^2}{2}\right) z^2 + \cdots \right]
\end{equation}
and
\begin{equation}\label{c7e2.6}
v(z) : = \frac{q(z) - 1}{q(z) + 1} = \frac{1}{2}\left[ q_{1} z + \left(q_{2} - \frac{q_{1}^2}{2}\right) z^2 + \cdots \right].
\end{equation}
Then $p(z)$ and $q(z)$ are analytic in $\triangle$ with $p(0) = 1 = q(0).$ Since $u, v : \triangle \rightarrow \triangle,$
the functions $p(z)$ and $q(z)$ have a positive real part in $\triangle,$ $|p_{i}| \leq 2$ and $|q_{i}| \leq 2.$
\\
\par Using \eqref{c7e2.5} and \eqref{c7e2.6} in \eqref{c7e2.3} and \eqref{c7e2.4} respectively, we have
\begin{equation}\label{c7e2.7}
(1 - \lambda) \frac{\mathcal{D}_q^{k + 1} f(z)} { \mathcal{D}_q^{k} f(z)}+ \lambda \frac{\mathcal{D}_q^{k + 2} f(z)}{\mathcal{D}_q^{k + 1} f(z)} = \phi\left( \frac{1}{2}\left[ p_{1} z + \left(p_{2} - \frac{p_{1}^2}{2}\right) z^2 + \cdots \right] \right)
\end{equation}
and
\begin{equation}\label{c7e2.8}
(1 - \lambda) \frac{\mathcal{D}_q^{k + 1} g(w)} { \mathcal{D}_q^{k} g(w)}+ \lambda \frac{\mathcal{D}_q^{k + 2} g(w)}{\mathcal{D}_q^{k + 1} g(w)} = \phi\left( \frac{1}{2}\left[ q_{1} w + \left(q_{2} - \frac{q_{1}^2}{2}\right) w^2 + \cdots \right] \right).
\end{equation}
In light of \eqref{c7e1} - \eqref{c7e2a}, and  from \eqref{c7e2.7} and \eqref{c7e2.8},
we have
\begin{multline*}
1 + (1 + \lambda)[2]_q^{k} a_{2} z + [2 (1 + 2 \lambda )[3]_q^{k}a_{3} - (1 +3\lambda)[2]_q^{2k}a_{2}^2] z^2 + \cdots\\
= 1 + \frac{1}{2} B_{1} p_{1} z + \left[\frac{1}{2} B_{1} (p_{2} - \frac{p_{1}^2}{2}) + \frac{1}{4} B_{2} p_{1}^2\right] z^2 + \cdots
\qquad \qquad \qquad
\end{multline*}
and
\begin{multline*}
1 - (1 + \lambda)[2]_q^{k}a_{2} w + \{[(8\lambda+4)[3]_q^{k} - (3\lambda +1)[2]_q^{2k})] a_{2}^2-2 (1 + 2 \lambda)[3]_q^{k}a_{3}\} w^2 + \cdots\\
= 1 + \frac{1}{2} B_{1} q_{1} w +\left [\frac{1}{2} B_{1} (q_{2} - \frac{q_{1}^2}{2}) + \frac{1}{4} B_{2} q_{1}^2\right] w^2 + \cdots
\qquad \qquad \qquad
\end{multline*}
which yields the following relations:

\begin{eqnarray}
 (1 + \lambda) [2]_q^{k} a_{2}  &=&  \frac{1}{2} B_{1} p_{1}\label{c7e3.5}\\
 -(1 + 3\lambda) [2]_q^{2 k} a_{2}^2 + 2(1 + 2 \lambda)[3]_q^{k} a_{3} &=& \frac{1}{2} B_{1} (p_{2} - \frac{p_{1}^2}{2}) \nonumber\\ &\qquad +& \frac{1}{4} B_{2} p_{1}^2 \label{c7e3.6}\\
  - (1 + \lambda) [2]_q^{k} a_{2}   &=&  \frac{1}{2} B_{1} q_{1}  \label{c7e3.7}
\end{eqnarray}
and
\begin{equation}\label{c7e3.8}
( 4 (1 + 2 \lambda)[3]_q^{k} -(1 + 3\lambda)[2]_q^{2 k} )a_{2}^2 - 2(1 + 2 \lambda)[3]_q^{k} a_{3} = \frac{1}{2} B_{1} (q_{2} - \frac{q_{1}^2}{2}) + \frac{1}{4} B_{2} q_{1}^2.
\end{equation}
From \eqref{c7e3.5} and \eqref{c7e3.7} it follows that
\begin{equation}\label{c7e3.9}
p_{1} = -q_{1}
\end{equation}
and
\begin{equation}\label{c7e3.10}
8 (1 + \lambda)^2 [2]_q^{2 k } a_{2}^2 = B_{1}^2 (p_{1}^2 + q_{1}^2).
\end{equation}\label{a2}
From \eqref{c7e3.6}, \eqref{c7e3.8} and \eqref{c7e3.10}, we obtain
\begin{equation}\label{a22}
a_{2}^2 = \frac{B_{1}^3 (p_{2} + q_{2}) }{4 [\{ 2(1 + 2 \lambda)[3]_q^{k} - (1 + 3\lambda)[2]_q^{k} \}B_{1}^2
+ (1 + \lambda)^2 (B_{1} - B_{2})[2]_q^{2 k}]}.
\end{equation}
Applying Lemma \ref{c7lem1.2} to the coefficients $p_{2}$ and $q_{2},$ we have
\begin{equation}
|a_{2}| \leq \frac{ B_{1} \sqrt{B_{1}} } { \sqrt{|\{2(1+2\lambda) [3]_q^k - (1 + 3\lambda)[2]_q^{2 k}\} B_{1}^2
+ (1 + \lambda)^2 (B_{1} - B_{2})[2]_q^{2 k}|} }
\end{equation}
By subtracting \eqref{c7e3.8} from \eqref{c7e3.6} and using \eqref{c7e3.9} and \eqref{c7e3.10}, we get
\begin{equation}\label{a3}
a_{3} = \frac{B_{1}^2 (p_{1}^2 + q_{1}^2)}{8 (1 + \lambda)^2 [2]_q^{2 k}} + \frac{B_{1} (p_{2} - q_{2})}{8 (1 + 2 \lambda) [3]_q^{k}}.
\end{equation}
Applying Lemma \ref{c7lem1.2} once again to the coefficients $p_{1}, p_{2}, q_{1}$ and $q_{2},$ we get
\begin{equation}
|a_{3}| \leq   \frac{B_{1}}{2(1 + 2 \lambda)[3]^{k}}+ \left(\frac{B_{1}}{(1 + \lambda)[2]^{k}}\right)^2.
\end{equation}
\end{proof}

\begin{remark}\label{c7rem3.6}
If $f\in \mathcal{M}{\Sigma}^k_q(\lambda, \left(\frac{1 + z}{1 - z} \right)^{\alpha})$ then,
we have the following estimates for the coefficients $|a_2|$
and $|a_3|:$
\begin{equation*}
|a_{2}| \leq \frac{2 \alpha}{\sqrt{| \{2(1 + 2 \lambda)[3]_q^{k} - (1 + 3\lambda)[2]^{2 k}  \}\alpha + (1-\alpha)(1 + \lambda)^2 [2]_q^{2 k} |}}
\end{equation*}
and
\begin{equation*}
|a_{3}| \leq \frac{4 \alpha^2}{(1 + \lambda)^2[2]_q^{2 k}} + \frac{\alpha}{(1 + 2 \lambda)[3]_q^{k}}.
\end{equation*}
For functions $f\in \mathcal{M}{\Sigma}_q^k(\lambda, \frac{1 + (1 - 2 \beta) z}{1 - z}),$  the inequalities \eqref{c7e3.3}
and \eqref{c7e3.4} yields the following estimates
\begin{equation*}
|a_{2}| \leq \sqrt{\frac{2 (1 - \beta)}{| 2 (1 + 2 \lambda)[3]_q^{k} - (1 +3 \lambda)[2]_q^{2 k} |}}
\end{equation*}
and
\begin{equation*}
|a_{3}| \leq \frac{4 (1 - \beta)^2}{(1 + \lambda)^2[2]_q^{2 k}} + \frac{(1 - \beta)}{(1 + 2 \lambda)[3]_q^{k}}.
\end{equation*}\end{remark}
\begin{remark}Consequently, when
$\lambda = 0$ and $\lambda =1$ one has  the estimates for the classes $\mathcal{S}{\Sigma}_q^{k}(\alpha),~~
 \mathcal{S}{\Sigma}_q^{k}(\beta)$ and $\mathcal{K}{\Sigma}^k_q(\alpha),~~ \mathcal{K}{\Sigma}^k_q(\beta)$
respectively.We note that, for  $\lim_{q\rightarrow 1^-}$  and for $k=0$ these estimates coincides with the results stated in \cite{xi}.\end{remark}
%##################################################################################################3
From Remark \ref{gms3.3}, Theorem \ref{c7thm3.4}  yields the following corollary.
\begin{corollary}\label{c1}
Let $f$ given by \eqref{c7e1} be in the class $\mathcal{M}{\Sigma}_q(\lambda, \phi).$ Then
\begin{equation}\label{ce1}
|a_{2}| \leq \frac{ B_{1} \sqrt{B_{1}} } { \sqrt{|(2(1+2\lambda) [3]_q - (1 + 3\lambda)[2]_q) B_{1}^2
+ (1 + \lambda)^2 (B_{1} - B_{2})[2]_q |} }
\end{equation}
and
\begin{equation}\label{ce2}
|a_{3}| \leq   \frac{B_{1}}{2(1 + 2 \lambda) [3]_q}+ \left(\frac{B_{1}}{(1 + \lambda)[2]_q}\right)^2
\end{equation}
\end{corollary}

%====================================================================================================================

\par  In the following section due to Frasin and Aouf \cite{BAF-MKA} and Panigarhi and Murugusundaramoorthy \cite{GMS-TP} we define the following new subclass involving the S\u al\u agean operator \cite{salagean}.
\section{Bi-Univalent function class $\mathcal{F}{\Sigma}_q^{k}(\mu, \phi)$ }
\begin{definition}\label{kva3.1}
For $0\leq \mu \leq 1,$ a  function $f \in \Sigma$  of the form (\ref{c7e1}) is said to be in the class $\mathcal{F}{\Sigma}_q^{k}(\mu, \phi)$
if the following subordination hold:
\begin{equation}\label{kva3.1a}
(1 - \mu) \frac{\mathcal{D}_q^{k } f(z)}{z}+ \mu (\mathcal{D}_q^{k } f(z))' \prec \phi(z)
\end{equation}
and
\begin{equation}\label{kva3.2}
(1 - \mu)\frac
{ \mathcal{D}_q^{k } g(w)}{w} + \mu (\mathcal{D}_q^{k } g(w))'\prec \phi(w)
\quad
\end{equation}where $z,w \in \Delta$, $g$ is given by (\ref {c7e2}) and $D^{k }_q f(z)$ is given by (\ref{dk}).
\end{definition}

\begin{remark}\label{kva3.2a} Suppose $f(z)\in\Sigma.$ If $\mu = 0,$ then
$\mathcal{F}{\Sigma}_q^{k}(0, \phi)\equiv\mathcal{H}{\Sigma}_q^{k}(\phi):$ thus,
 $f\in \mathcal{H}{\Sigma}_q^{k}(\phi)$ if the following subordination holds:
\begin{equation*}
\frac{\mathcal{D}_q^{k } f(z)}{z} \prec \phi(z)~~{\rm and}~~ \frac{\mathcal{D}_q^{k}g(w)}{w}  \prec \phi(w)
\end{equation*}where $z,w \in \Delta$ and $g$ is given by (\ref {c7e2}).
\end{remark}
\begin{remark}\label{c7rem3.3}
Suppose $f(z)\in\Sigma.$ If $\mu = 1$, then
$\mathcal{F}{\Sigma}_q^{k}(1, \phi)\equiv\mathcal{P}{\Sigma}_q^{k}(\phi):$ thus,
 $f\in \mathcal{P}{\Sigma}_q^{k}(\phi)$ if the following subordination holds:
\begin{equation*}
(\mathcal{D}_q^{k}f(z))'\prec \phi(z)~~{\rm and}~~ (\mathcal{D}_q^{k} g(w))' \prec \phi(w)
\end{equation*}where $z,w \in \Delta$ and $g$ is given by (\ref {c7e2}).
\end{remark}
\par It is of interest to  note that $\mathcal{F}{\Sigma}_q^{0}(\mu, \phi) = \mathcal{F}{\Sigma}_q(\mu,\phi)$ if the following subordination hold:
\begin{equation}\label{kva3.1af}
(1 - \mu) \frac{f(z)}{z}+ \mu (\mathcal{D}_qf(z)) \prec \phi(z)
\end{equation}
and
\begin{equation}\label{kva3.2f}
(1 - \mu)\frac
{g(w)}{w} + \mu (\mathcal{D}_q g(w))\prec \phi(w)
\quad
\end{equation}where $z,w \in \Delta$, $g$ is given by (\ref {c7e2}) and $D_q^{k } f(z)$ is given by (\ref{dk}).

%=========================================================================================================================
\begin{theorem}\label{thm3.4}
Let $f$ given by \eqref{c7e1} be in the class $\mathcal{F}{\Sigma}_q^{k}(\mu, \phi).$ Then
\begin{equation}\label{7e3.3}
|a_{2}| \leq \frac{ B_{1} \sqrt{B_{1}} } { \sqrt{|(1+2\mu) [3]_q^k  B_{1}^2
+ (1 + \mu)^2 [2]_q^{2 k} (B_{1} - B_{2})|} }
\end{equation}
and
\begin{equation}\label{7e3.4}
|a_{3}| \leq B_{1} \left(  \frac{B_{1}}{(1 + \mu)^2 [2]_q^{2k}} + \frac{1}{(1 + 2 \mu) [3]_q^{k}}  \right).
\end{equation}
\end{theorem}

\begin{proof}
Proceeding as in the proof of Theorem \ref{c7thm3.4} we can arrive the following relations.
\begin{eqnarray}
 (1 + \mu) [2]_q^{k} a_{2}  &=&  \frac{1}{2} B_{1} p_{1}\label{7e3.5}\\
 (1 + 2 \mu) [3]_q^{k} a_{3} &=& \frac{1}{2} B_{1} (p_{2} - \frac{p_{1}^2}{2}) + \frac{1}{4} B_{2} p_{1}^2 \label{e3.6}\\
  - (1 + \mu) [2]_q^{k} a_{2}   &=&  \frac{1}{2} B_{1} q_{1}  \label{7e3.7}
\end{eqnarray}
and
\begin{equation}\label{7e3.8}
 2(1 + 2 \mu) [3]_q^{k} a_{2}^2 - (1 + 2 \mu) [3]_q^{k} a_{3} = \frac{1}{2} B_{1} (q_{2} - \frac{q_{1}^2}{2}) + \frac{1}{4} B_{2} q_{1}^2.
\end{equation}
From \eqref{7e3.5} and \eqref{7e3.7} it follows that
\begin{equation}\label{7e3.9}
p_{1} = -q_{1}
\end{equation}
and
\begin{equation}\label{7e3.10}
8(1 + \mu)^2 [2]_q^{2 k} a_{2}^2 = B_{1}^2 (p_{1}^2 + q_{1}^2).
\end{equation}
From \eqref{e3.6}, \eqref{7e3.8} and \eqref{7e3.10}, we obtain
\begin{equation*}
a_{2}^2 = \frac{B_{1}^3 (p_{2} + q_{2}) }{4 [(1 + 2 \mu) [3]_q^{k}B_{1}^2
+ (B_{1} - B_{2})(1 + \mu)^2 [2]_q^{2 k}]}.
\end{equation*}
Applying Lemma \ref{c7lem1.2}~~to the coefficients $p_{2}$ and $q_{2},$ we immediately get the
desired estimate on $|a_{2}|$ as asserted in \eqref{7e3.3}.

By subtracting \eqref{7e3.8} from \eqref{e3.6} and using \eqref{7e3.9} and \eqref{7e3.10}, we get
\begin{equation*}
a_{3} = \frac{B_{1}^2 (p_{1}^2 + q_{1}^2)}{8 (1 + \mu)^2 [2]_q^{2 k}} + \frac{B_{1} (p_{2} - q_{2})}{4 (1 + 2 \mu) [3]_q^{k}}.
\end{equation*}
Applying Lemma \ref{c7lem1.2} to the coefficients $p_{1}, p_{2}, q_{1}$ and $q_{2},$ we get
the desired estimate on $|a_{3}|$ as asserted in \eqref{7e3.4}.
\end{proof}
\begin{remark}Consequently, when
$\mu = 0$ and $\mu =1$  and by taking $\phi$ as in \eqref{phi01} and \eqref{phi02}one can deduce the estimates for the classes $\mathcal{H}{\Sigma}_q^{*}(\alpha), \mathcal{H}{\Sigma}_q^{*}(\beta)$ easily. We note that, for $k = 0$ and $\lim_q\rightarrow 1^-$ these estimates coincides with the results stated in \cite{BAF-MKA,DEN,srivastava}.
\end{remark}
\section{Fekete-Szeg$\ddot{o}$ inequalities }
Making use of the values of $a_2^2$ and $a_3$,  and motivated by the  recent work of Zaprawa \cite{zap} we prove the following Fekete-Szeg$\ddot{o}$ result .
\begin{theorem}\label{t2}
Let the function $f(z)\in\mathcal{M}{\Sigma}_q^{k}(\lambda, \phi)$ and $\tau \in \mathbb{C},$ then
\begin{equation}\label{bi-th2-b-a2-a3}
\qquad  |a_3-\tau a_2^2|\leq \left\{
\begin{array}{ll}
\frac{B_1}{2(1+2\lambda)[3]_q^k},\qquad  & 0\leq
|\Theta(\tau)| <\frac{1}{8(1+2\lambda)[3]_q^k}{,}\quad
\qquad \qquad \quad  \\
&  \\
4 B_1|\Theta(\tau)|, & |\Theta(\tau)| \geq \frac{1}{8(1+2\lambda)[3]_q^k}{,}
\end{array}
\right.
\end{equation}
\end{theorem}
\begin{proof}
From (\ref{a3})  we have
$
a_{3} = a_2^2 + \frac{B_{1} (p_{2} - q_{2})}{8 (1 + 2 \lambda) [3]_q^{k}}.
$
Using (\ref{a22}),
\begin{multline*}a_3-\tau a_2^2= \frac{B_{1} (p_{2} - q_{2})}{8 (1 + 2 \lambda) [3]_q^{k}}+(1-\tau)a_2^2 \\=\frac{B_{1} (p_{2} - q_{2})}{8 (1 + 2 \lambda) [3]_q^{k}}\\+(1- \tau)\left(
\frac{B_{1}^3 (p_{2} + q_{2}) }{4 [\{ 2(1 + 2 \lambda)[3]_q^{k} - (1 + 3\lambda)[2]_q^{k} \}B_{1}^2
+ (1 + \lambda)^2 (B_{1} - B_{2})[2]_q^{2 k}]}\right)
\end{multline*}
by simple calculation we get
\begin{multline*}
 a_3 - \tau a_2^2 = B_1 \left[ \left(\Theta(\tau) + \frac{1}{8 (1 + 2 \lambda) [3]_q^{k}}\right) p_2
+ \left( \Theta(\tau) - \frac{1}{8 (1 + 2 \lambda) [3]_q^{k}}\right) q_2\right],
\end{multline*}
where
$$ \Theta(\tau) = \frac{ B_1^2(1-\tau)}{4 [\{ 2(1 + 2 \lambda)[3]_q^{k} - (1 + 3\lambda)[2]_q^{k} \}B_{1}^2
+ (1 + \lambda)^2 (B_{1} - B_{2})[2]_q^{2 k}]}.$$
Since all $B_j$ are real and $B_1 > 0$, we have
\begin{equation*}
\qquad  |a_3-\tau a_2^2|\leq 2 B_1 \left| \left(\Theta(\tau) + \frac{1}{8 (1 + 2 \lambda) [3]_q^{k}}\right)
+ \left( \Theta(\tau) - \frac{1}{8 (1 + 2 \lambda) [3]_q^{k}}\right)\right|,
\end{equation*}
which completes the proof.
\end{proof}
   \par Proceeding as in above theorem one can easily prove the following result for $f(z)\in\mathcal{F}{\Sigma}_q^{k}(\mu, \phi)$ hence we state the following without proof.
\begin{theorem}\label{t2f}
Let the function $f(z)\in\mathcal{F}{\Sigma}_q^{k}(\mu, \phi)$ and $\tau \in \mathbb{C},$ then
\begin{equation*}
|a_3-\tau a_2^2|\leq 2 B_1 \left| \left(\Phi(\tau) + \frac{1}{4 (1 + 2 \mu) [3]_q^{k}}\right)
+ \left( \Phi(\tau) - \frac{1}{4 (1 + 2 \mu) [3]_q^{k}}\right)\right|
\end{equation*}
where
$$ \Phi(\tau) = \frac{ B_1^2(1-\tau)}{4 [(1 + 2 \mu) [3]_q^{k}B_{1}^2
+ (B_{1} - B_{2})(1 + \mu)^2 [2]_q^{2 k}]}.$$
\end{theorem}

{\bf Concluding Remarks:} Taking $\lambda =0 ~(and ~1)$ in Theorem \ref{t2}, we can state the Fekete-Szeg$\ddot{o}$ inequality for the function class $\mathcal{S}{\Sigma}_q^{k}(\phi). (and ~\mathcal{K}{\Sigma}_q^{k}(\phi)$ respectively. Putting $\mu =0 ~(and ~1)$ in Theorem \ref{t2f}, we can state the Fekete-Szeg$\ddot{o}$ inequality for the function class $\mathcal{H}{\Sigma}_q^{k}(\phi). (and ~\mathcal{P}{\Sigma}_q^{k}(\phi)$ respectively.\\ 
{\bf Future Work:} Making use of the values of $a_2$ and $a_3$, and finding $a_4$ we can caluculate Hankel determinant coefficient for the function classes.

%=========================================================================================================================
%------------------------

\end{document}